\documentclass[a4paper, 11pt]{article}
\usepackage[utf8]{inputenc}
\usepackage[T1]{fontenc}

\usepackage[
	sorting=nyt,
	backend = biber,
	hyperref = auto,
	backref = true,
	doi = false,
	url = false
]{biblatex}

\usepackage{amsthm}
\usepackage{amssymb}
\usepackage{mathtools}
\usepackage{stmaryrd}	\SetSymbolFont{stmry}{bold}{U}{stmry}{m}{n}	% To avoid warnings
\usepackage{bm}
\usepackage{nccrules}
\usepackage{euscript}
\usepackage{mathrsfs}
\usepackage{upgreek}
\usepackage{tikz-cd}
\usetikzlibrary{positioning, shapes.geometric, patterns, graphs, decorations.pathmorphing}

\usepackage{paralist}

\usepackage[colorlinks]{hyperref}
\hypersetup{
	linkcolor=blue,
	citecolor=red,
	urlcolor=teal,
	pdftitle = {Variation of additive characters in the local Langlands correspondence of Mp(2n)},
	pdfauthor = {Wen-Wei Li}
}

\newcommand{\Z}{\ensuremath{\mathbb{Z}}}

\newcommand{\R}{\ensuremath{\mathbb{R}}}
\newcommand{\CC}{\ensuremath{\mathbb{C}}}

\newcommand{\Tr}{\operatorname{tr}}
\newcommand{\Weil}[1]{\ensuremath{\mathrm{Weil}_{#1}}}	% The Weil group

\newcommand{\dd}{\mathop{}\!\mathrm{d}}
\newcommand{\mes}{\operatorname{mes}}	% Measure

\newcommand{\lrangle}[1]{\ensuremath{\langle #1 \rangle}}
\newcommand{\sgn}{\ensuremath{\mathrm{sgn}}}

\newcommand{\identity}{\ensuremath{\mathrm{id}}}

\newcommand{\rightiso}{\ensuremath{\stackrel{\sim}{\rightarrow}}}

\newcommand{\Hm}{\operatorname{H}}	% Homology
\newcommand{\Ker}{\operatorname{ker}}

\newcommand{\Image}{\operatorname{im}}

\newcommand{\Ad}{\operatorname{Ad}}

\newcommand{\Gm}{\ensuremath{\mathbb{G}_\mathrm{m}}}

\newcommand{\GL}{\operatorname{GL}}

\newcommand{\SO}{\operatorname{SO}}
\newcommand{\Or}{\operatorname{O}}

\newcommand{\SL}{\operatorname{SL}}
\newcommand{\Sp}{\operatorname{Sp}}
\newcommand{\GSp}{\operatorname{GSp}}
\newcommand{\Lgrp}[1]{\ensuremath{{}^{\mathrm{L}} #1}}	% The L-group
\newcommand{\Mp}{\ensuremath{\widetilde{\mathrm{Sp}}}}
\newcommand{\MMp}{\operatorname{Mp}}

\newcommand{\bmu}{\ensuremath{\bm\mu}}
\newcommand{\bpsi}{{\ensuremath{\uppsi}}}
\newcommand{\Endo}{\ensuremath{\mathcal{E}}}
\newcommand{\orbI}{\ensuremath{\mathcal{I}}}

\newcommand{\elli}{\operatorname{ell}}
\newcommand{\asp}{\ensuremath{\dashrule[.7ex]{2 2 2 2}{.4}}} % Le symbole pour anti-spécifiques
\newcommand{\rev}{\ensuremath{\mathbf{p}}} % Le symbole pour revêtements
\newcommand{\Trans}{\ensuremath{\mathcal{T}}}	% Transfert géométrique
\newcommand{\trans}{\ensuremath{\check{\mathcal{T}}}}	% Transfert de distributions

\theoremstyle{plain}
\newtheorem{proposition}{Proposition}
\newtheorem{lemma}[proposition]{Lemma}
\newtheorem{theorem}[proposition]{Theorem}
\newtheorem{corollary}[proposition]{Corollary}
\theoremstyle{definition}
\newtheorem{definition}[proposition]{Definition}
\newtheorem{definition-theorem}[proposition]{Definition--Theorem}
\newtheorem{definition-proposition}[proposition]{Definition--Proposition}
\newtheorem{remark}[proposition]{Remark}

\theoremstyle{plain}

\theoremstyle{definition}

\numberwithin{equation}{section}
\numberwithin{proposition}{subsection}
\numberwithin{Thm}{section}	% In the Introduction only
\numberwithin{Rmk}{section}	% In the Introduction only

\usepackage{amsmath}
\usepackage[nottoc]{tocbibind}

\usepackage{geometry}
\title{Variation of additive characters in the transfer for $\MMp(2n)$}
\author{Wen-Wei Li}
\date{}

\DeclareFieldFormat{postnote}{#1}	% Disable pagination in postnote for biblatex
\makeatletter
\renewcommand{\l@section}{\@dottedtocline{1}{1.5em}{2.0em}}
\renewcommand{\l@subsection}{\@dottedtocline{2}{4.0em}{3.0em}}
\makeatother

\begin{document}
	
\maketitle

\begin{abstract}
	Let $\mathrm{Mp}(2n)$ be the metaplectic group of rank $n$ over a local field $F$ of characteristic zero. In this note, we determine the behavior of endoscopic transfer for $\mathrm{Mp}(2n)$ under variation of additive characters of $F$. The arguments are based on properties of transfer factor, requiring no deeper results from representation theory. Combined with the endoscopic character relations of Luo, this provides a simple and uniform proof of a theorem of Gan--Savin, which describes how the local Langlands correspondence for $\mathrm{Mp}(2n)$ depends on the additive characters.
\end{abstract}
	
{\scriptsize
	\begin{tabular}{ll}
		\textbf{MSC (2020)} & Primary 22E50; Secondary 11F70, 11F72 \\
		\textbf{Keywords} & endoscopy, metaplectic group, local Langlands correspondence
	\end{tabular}}

\setcounter{tocdepth}{1}
\tableofcontents
	
\section{Introduction}\label{sec:intro}
\subsection{Overview}
Let $F$ be a local field of characteristic zero. Consider the symplectic group $\Sp(W)$ associated with a symplectic $F$-vector space $(W, \lrangle{\cdot|\cdot})$ of dimension $2n$. Put $\bmu_m := \{z \in \CC^{\times}: z^m = 1\}$ for all $m$. The metaplectic covering is a central extension of topological groups
\begin{equation*}
	1 \to \bmu_2 \to \MMp(W) \xrightarrow{\rev} \Sp(W) \to 1.
\end{equation*}
It plays important roles in various scenarios, such as the $\Theta$-correspondence, and there is also an adélic counterpart over number fields of great arithmetic interest.

Weil's original construction of $\Mp(W)$ involves $(W, \lrangle{\cdot|\cdot})$ as well as a chosen additive character $\bpsi: F \to \CC^{\times}$ (unitary, non-trivial). Later on, $\Mp(W)$ is characterized as the unique non-trivial twofold coverings of $\Sp(W)$ when $F \neq \CC$, up to unique isomorphisms, and it splits uniquely when $F = \CC$; see \cite[Theorem 10.4]{Mo68}.

In the literature, it is customary to write $\Sp(W) = \mathrm{Sp}(2n)$ and $\MMp(W) = \mathrm{Mp}(2n)$. We are interested in the genuine representations of $\MMp(W)$, i.e.\ those representations on which $z \in \bmu_2$ acts as $z \cdot \identity$. Although $\MMp(W)$ is not the group of $F$-points of some connected reductive $F$-group, Gan and Savin \cite{GS1} used $\Theta$-correspondences to obtain a local Langlands correspondence (LLC) for $\MMp(W)$ when $F$ is non-Archimedean, in which $\Sp(2n, \CC)$ plays the role of the Langlands dual group of $\MMp(W)$. The Archimedean counterpart is due to Adams and Barbasch \cite{AB98}.
	
Later on, based on endoscopy for metaplectic groups \cite{Li11} and the global multiplicity formula of Gan--Ichino \cite{GI18}, C.\ Luo \cite{Luo20} proved the endoscopic character relations for $\MMp(W)$ and characterized the correspondences above in terms of endoscopic transfer.

For $\MMp(W)$, the LLC and endoscopic transfer both depend on the choices of $\lrangle{\cdot|\cdot}$ and $\bpsi$; more precisely, on $\bpsi \circ \lrangle{\cdot|\cdot}$. The dependence is elucidated completely by \cite[Theorem 12.1]{GS1} for non-Archimedean $F$, to be reviewed in Theorem \ref{prop:GS-preview}. The proof in \textit{loc.\ cit.}\ is surprisingly roundabout: besides advanced properties of $\Theta$-lifting, it also used the local Gross--Prasad conjecture for special orthogonal groups. On the other hand, the Archimedean case should be contained in \cite{AB98}, or follows from similar arguments, but there seems to be no written account.

The aim of this note is to offer a direct endoscopic proof of the aforementioned result, which applies uniformly to all $F$. Given Luo's character relations, we will deduce it from a variation formula for the endoscopic transfer for orbital integrals.
	
\subsection{Main results and proofs}\label{sec:main-result}
Let $G := \Sp(W)$ and $\tilde{G}^{(2)} := \MMp(W)$. Recall that $\tilde{G}^{(2)}$ is independent of $\bpsi$, up to unique isomorphisms. On the other hand, in the theory of endoscopy, it is more convenient to enlarge it to an eightfold covering $\tilde{G}$ by pushing out through $\bmu_2 \hookrightarrow \bmu_8$. All constructions on the level of $\tilde{G}$ will depend on $\bpsi$ and $\lrangle{\cdot|\cdot}$.

A representation of $\tilde{G}$ is said to be genuine if $z \in \bmu_8$ acts as $z \cdot \identity$. The study of genuine representations of $\tilde{G}$ is equivalent to that of $\tilde{G}^{(2)}$.

For every bounded L-parameter $\phi$ for $\tilde{G}$, defined by postulating
\[ \tilde{G}^\vee := \Sp(2n, \CC) = \SO(2n+1)^\vee, \]
we set $S_\phi := Z_{\tilde{G}^\vee}(\Image(\phi))$ and let $\EuScript{S}_\phi$ be its component group, which is finite abelian. Let $\EuScript{S}_\phi^\vee$ be the Pontryagin dual of $\EuScript{S}_\phi$. For all $\chi \in \EuScript{S}_\phi^\vee$, the LLC in \cite{GS1} gives a tempered genuine irreducible representation $\pi_{\phi, \chi}$ of $\tilde{G}$. An exponent $\bpsi$ indicates its dependence on additive characters.

\begin{theorem}[= Theorem \ref{prop:GS}]\label{prop:GS-preview}
	Given $c \in F^{\times}$, define the additive character $\bpsi_c$ of $F$ by $\bpsi_c(x) = \bpsi(cx)$.
	For all bounded L-parameter $\phi$ and $\chi \in \EuScript{S}_\phi^\vee$, we have
	\[ \pi^{\bpsi_c}_{\phi\zeta, \chi \delta_c} \simeq \pi^{\bpsi}_{\phi, \chi}, \]
	where
	\begin{itemize}
		\item $\zeta: \Weil{F} \to \bmu_2 \simeq Z_{\tilde{G}^\vee}$ is the homomorphism attached to the coset $cF^{\times 2}$ by local class field theory, which can be used to twist L-parameters, so that $S_\phi = S_{\zeta\phi}$;
		\item $\delta_c \in \EuScript{S}_\phi^\vee$ is explicitly defined in terms of local root numbers and $\zeta(-1)$ (Definition \ref{def:delta-c}).
	\end{itemize}
\end{theorem}

The above recovers \cite[Theorem 12.1]{GS1}. This statement appeared first as \cite[Conjecture 11.3]{GGP1}. We shall deduce it from the following result about endoscopic transfer.

Elliptic endoscopic data $\mathbf{G}^!$ of $\tilde{G}$ are in bijection with pairs $(n', n'') \in \Z_{\geq 0}^2$ with $n' + n'' = n$. The corresponding endoscopic group is
\[ G^! = \SO(2n'+1) \times \SO(2n''+1) \]
where the $\SO$ groups are split; note that this description is insensitive to $\bpsi$. Denote the endoscopic transfer of orbital integrals from $\tilde{G}$ to $G^!(F)$ by $\Trans_{\mathbf{G}^!, \tilde{G}}$. The existence of transfer is the main result of \cite{Li11}, and the endoscopic character relations are given by its dual (i.e.\ transpose) $\trans_{\mathbf{G}^!, \tilde{G}}$. Transfer can be defined on the level of $\tilde{G}^{(2)}$, and we add an exponent $\bpsi$ to indicate its dependence on additive characters.

Denote by $(\cdot, \cdot)_{F, 2}$ the quadratic Hilbert symbol on $F^{\times} \times F^{\times}$.

\begin{theorem}[= Theorem \ref{prop:Trans-var}]\label{prop:Trans-var-preview}
	Let $c \in F^{\times}$ and define the corresponding $\zeta: \Weil{F} \to \bmu_2$ as before. Let $\mathbf{G}^!$ be an elliptic endoscopic datum corresponding to $(n', n'')$. Then
	\begin{align*}
		\Upsilon^\zeta \circ \Trans^{\bpsi_c}_{\mathbf{G}^!, \tilde{G}^{(2)}} & = (c, -1)_{F, 2}^{n''} \Trans^{\bpsi}_{\mathbf{G}^!, \tilde{G}^{(2)}}, \\
		\trans^{\bpsi_c}_{\mathbf{G}^!, \tilde{G}^{(2)}} \circ \Upsilon_\zeta & = (c, -1)_{F, 2}^{n''} \trans^{\bpsi}_{\mathbf{G}^!, \tilde{G}^{(2)}},
	\end{align*}
	where $\Upsilon^{\zeta}$ is the involution on the space of stable orbital integrals on $G^!(F)$ by multiplication by the character
	\begin{align*}
		s^!_c: G^!(F) & \to \bmu_2 \\
		(\gamma', \gamma'') & \mapsto (c, \mathrm{SN}(\gamma'))_{F, 2} (c, \mathrm{SN}(\gamma''))_{F, 2},
	\end{align*}
	with $\mathrm{SN}$ being the spinor norm, and $\Upsilon_\zeta$ is its dual.
\end{theorem}

The description of $\Upsilon_\zeta$ in terms of bounded L-parameters is straightforward: it is translation by $\zeta$ on both $\SO$ factors. See Corollary \ref{prop:Upsilon-effect}.

These results are used in \cite[\S\S 9--10]{Li24a} to describe certain Arthur packets of $\tilde{G}$.

Luo's character relation is applied to deduce Theorem \ref{prop:GS-preview} from Theorem \ref{prop:Trans-var-preview}, as performed in \cite[Proposition 9.2.3]{Li24a}. Luo's proof does not involve \cite[Theorem 12.1]{GS1}, and neither does \cite{GI18}, so there is no worry of circularity.

Below is a sketch of the proof of Theorem \ref{prop:Trans-var-preview}. It reduces to a property of transfer factors $\Delta$ (Proposition \ref{prop:Delta-var}). To prove the latter result, we pick any $g \in \GSp(W)$ with similitude factor $c$. The automorphism $\Ad(g)$ of $G(F)$ lifts to $\tilde{G}^{(2)}$ and $\tilde{G}$, relating the transfer factors for $\bpsi$ and $\bpsi_c$ by a transport of structure. To complete the proof, we need the following ingredients:
\begin{itemize}
	\item description of $\mathrm{SN}$ in terms of a convenient parametrization of stable conjugacy classes that is used in \cite{Li11} (Lemma \ref{prop:SN-formula});
	\item determine the relative position between $\tilde{\delta}$ and $\Ad(g)(\tilde{\delta})$ as an element of $\Hm^1(F, Z_G(\delta))$, for all $\tilde{\delta} \in \tilde{G}^{(2)}$ with regular semisimple image $\delta \in G(F)$.
\end{itemize}
For the first ingredient about spinor norms, we refer to \cite[\S 5.1]{FKS19}. The second ingredient is relatively subtle: $\Ad(g)(\tilde{\delta})$ must be calibrated by a sign to make it stably conjugate to $\tilde{\delta}$ in the sense of Adams (see \cite[\S 9.1]{Li20}). Fortunately, most of the required computations have been done in \cite{Li20}.

We then conclude by the cocycle property \cite[Proposition 5.13]{Li11} of metaplectic transfer factors.

The contents are organized as follows. In \S\S\ref{sec:review-endoscopy}--\ref{sec:review-LLC}, we give a summary about metaplectic groups, the metaplectic theory of endoscopy, LLC, and the endoscopic character relation. In \S\ref{sec:SN}, we collect the required properties of spinor norms and describe how the corresponding characters affect the L-parameters. In \S\ref{sec:var}, we prove Proposition \ref{prop:Delta-var}, Theorem \ref{prop:Trans-var} and then Theorem \ref{prop:GS} following the strategy sketched above. Moreover, in the final \S\ref{sec:remark-L}, we put these results in the context of L-groups for coverings, following Weissman \cite{Weis18} and Gan--Gao \cite{GG}, and discuss Prasad's conjecture on contragredients for $\tilde{G}^{(2)}$ briefly.

The author is grateful to Fei Chen and Caihua Luo for helpful conversations.

This research is supported by NSFC, Grant No.\ 11922101 and 12321001.

\subsection{Conventions}
All representations are realized on $\CC$-vector spaces. For every $m \in \Z_{\geq 1}$, we write $\bmu_m := \{z \in \CC^{\times}: z^m = 1 \}$.

For every group $\Gamma$ and $g \in \Gamma$, denote by $\Ad(g)$ the automorphism $\gamma \mapsto g\gamma g^{-1}$ of $\Gamma$.

Throughout this article, $F$ is a local field of characteristic zero unless otherwise specified. The Weil group of $F$ is denoted by $\Weil{F}$, and the local Langlands group of $F$ is
\[ \mathcal{L}_F := \begin{cases}
	 \Weil{F}, & F \;\text{is Archimedean} \\
	 \Weil{F} \times \SL(2, \CC), & F \;\text{is non-Archimedean.}
\end{cases}\]
The quadratic Hilbert symbol on $F^{\times} \times F^{\times}$ is denoted by $(\cdot, \cdot)_{F, 2}$.

An additive character of $F$ means a continuous, unitary and non-trivial homomorphism $\bpsi: F \to \CC^{\times}$. For $c \in F^{\times}$, we obtain $\bpsi_c$ with $\bpsi_c(x) = \bpsi(cx)$.

For an algebraic group $R$ over $F$, denote its center (resp.\ identity connected component) by $Z_R$ (resp.\ $R^{\circ}$). Denote the group of $F$-points of $R$ by $R(F)$. For all $\delta \in R(F)$, let $R^\delta$ be the centralizer of $\delta$ in $R$ and $R_\delta := (R^\delta)^{\circ}$.
	
Assume $R$ is a connected reductive $F$-group. Let $R_{\mathrm{reg}} \subset R$ (resp.\ $R_{\mathrm{sreg}} \subset R$) be the regular semisimple (resp.\ strongly regular semisimple) locus, which is Zariski open and dense; recall that $\delta \in R_{\mathrm{reg}}$ is said to be strongly regular if $R_\delta = R^\delta$. As a matter of fact, for semisimple and simply connected $R$ (eg.\ symplectic groups), we have $R_{\mathrm{reg}} = R_{\mathrm{sreg}}$, but this does not hold in general (eg.\ for special orthogonal groups).

The Langlands dual group $R^\vee$ of a connected reductive $F$-group $R$ is always taken over $\CC$. Denote the L-group of $R$ as $\Lgrp{R} = R^\vee \rtimes \Weil{F}$. The set of equivalence classes of L-parameters (resp.\ bounded L-parameters) of $R$ is denoted by $\Phi(R)$ (resp.\ $\Phi_{\mathrm{bdd}}(R)$).

Assume $F$ is any field with $\mathrm{char}(F) \neq 2$ hereafter. By a symplectic (resp.\ quadratic) $F$-vector space, we mean a pair $(W, \lrangle{\cdot|\cdot})$ (resp.\ $(V, q)$ or simply $V$) where $W$ (resp.\ $V$) is a finite-dimensional $F$-vector space and $\lrangle{\cdot|\cdot}$ (resp.\ $q$) is a non-degenerate alternating bilinear form (resp.\ quadratic form) on it; we have the symplectic (resp.\ orthogonal) group $\Sp(W)$ (resp.\ $\Or(V)$).

The notation $\SO(2n+1)$ will always mean the split odd special orthogonal group of rank $n$.

For a field $E$ and an étale $E$-algebra $L$, we denote the norm map by $N_{L|E}: L \to E$.

\section{Review of endoscopy}\label{sec:review-endoscopy}
Most of the materials below are taken from \cite{Li11}. See also \cite[Remarque 2.3.1]{Li15} for a (partial) erratum.
	
\subsection{Endoscopic data of metaplectic groups}\label{sec:endoscopic-data}
Consider a symplectic $F$-vector space $(W, \lrangle{\cdot|\cdot})$ of dimension $2n$, where $n \in \Z_{\geq 1}$, and set $G := \Sp(W)$. We will consider two kinds of coverings of $G(F)$.
	
First, when $F \neq \CC$, there exists a non-trivial twofold covering
\[ 1 \to \bmu_2 \to \tilde{G}^{(2)} \xrightarrow{\rev^{(2)}} G(F) \to 1, \]
which is unique up to unique isomorphisms in the category of central extension of locally compact groups; see \cite[Theorem 10.4]{Mo68}. When $F = \CC$, we let $\tilde{G}^{(2)}$ be the trivial twofold covering $\bmu_2 \times G(F)$. In both cases, we also denote $\tilde{G}^{(2)}$ as $\MMp(W)$ or $\MMp(2n)$.
	
Secondly, let $\tilde{G}$ be the push-out of $\tilde{G}^{(2)} \xrightarrow{\rev^{(2)}} G(F)$ via $\bmu_2 \hookrightarrow \bmu_8$. This yields a central extension of locally compact groups
\[ 1 \to \bmu_8 \to \tilde{G} \xrightarrow{\rev} G(F) \to 1, \]
and there is an injective homomorphism $\iota: \tilde{G}^{(2)} \to \tilde{G}$, compatible with the homomorphisms onto $G(F)$. Such a homomorphism $\iota$ is unique since any two homomorphisms with these properties differ by a character of $G(F)$, whereas $G(F)$ equals its own derived subgroup.

When working with the eightfold covering $\tilde{G}$, a non-trivial unitary character $\bpsi: F \to \CC^{\times}$ will always be fixed tacitly. The cocycle describing $\tilde{G}$ arises from Schrödinger models for the Weil representation $\omega_{\bpsi}$, and is cleaner than that of $\tilde{G}^{(2)}$; this stems ultimately from the fact that Weil's index $\gamma_{\bpsi}(\cdot)$ for quadratic $F$-vector spaces is $\bmu_8$-valued. We shall write $\tilde{G} = \tilde{G}^{\bpsi}$ when the role of $\bpsi$ is to be emphasized.

As $\tilde{G}$ is the push-out of $\tilde{G}^{(2)}$, the genuine representations (resp.\ genuine invariant distributions) of $\tilde{G}$ and $\tilde{G}^{(2)}$ are identified.
	
With the chosen $\bpsi$ and $\lrangle{\cdot|\cdot}$, define the Langlands dual group of $\tilde{G}$ as
\[ \tilde{G}^\vee := \Sp(2n, \CC) \quad \text{with trivial Galois action.} \]
	
The set of elliptic endoscopic data of $\tilde{G}$ is defined to be
\begin{gather*}
	\Endo_{\elli}(\tilde{G}) = \Endo_{\elli}(\tilde{G}^{(2)}) := \left\{ s \in \tilde{G}^\vee : s^2 = 1 \right\} \bigg/ \tilde{G}^\vee\text{-conj.}
\end{gather*}
Elements of $\Endo_{\elli}(\tilde{G})$ are in bijection with pairs $(n' ,n'') \in \Z_{\geq 0}^2$ such that $n' + n'' = n$. In fact, $2n'$ (resp.\ $2n''$) is the multiplicity of $+1$ (resp.\ $-1$) as an eigenvalue of $s$. For each $(n', n'')$, the corresponding endoscopic group is
\[ G^! := \SO(2n' + 1) \times \SO(2n'' + 1). \]
	
Following Arthur, the elliptic endoscopic data of $\tilde{G}$ will be written as $\mathbf{G}^!$, with $G^!$ being the underlying endoscopic group.
	
\subsection{Correspondence of conjugacy classes}\label{sec:conjugacy-class}
Next, we review the parametrization of conjugacy classes in $G_{\mathrm{reg}}(F)$ following \cite[\S 3]{Li11} or \cite[\S 3.1]{Li20}. Consider the data $(K, K^\natural, x, c)$ where
\begin{itemize}
	\item $K$ is an étale $F$-algebra of dimension $2n$, endowed with an involution $\tau$;
	\item $K^\natural = \{t \in K: \tau(t) = t \}$;
	\item $x \in K^{\times}$ satisfies $\tau(x) = x^{-1}$ and $K = F[x]$;
	\item $c \in K^{\times}$ satisfies $\tau(c) = -c$.
\end{itemize}
Two data $(K, K^\natural, x, c)$ and $(K_1, K_1^\natural, x_1, c_1)$ are said to be equivalent if there exists an isomorphism of $F$-algebras $\varphi: K \rightiso K_1$ that preserves the involutions, $\varphi(x) = x_1$, and $\varphi(c) \in c_1 \cdot N_{K_1 | K_1^\natural}(K_1^{\times})$.
	
There is a natural bijection $\mathcal{O}$ from the set of equivalence classes of data $(K, K^\natural, x, c)$ onto $G_{\mathrm{reg}}(F) / \text{conj}$. To parameterize stable conjugacy classes in $G_{\mathrm{reg}}(F)$, we simply forget the datum $c$ and consider equivalence classes of $(K, K^\natural, x)$ instead.	
	
Let $H := \SO(2n+1)$. A similar parametrization applies to conjugacy classes in $H_{\mathrm{sreg}}(F)$, the only difference being that one considers data $(K, K^\natural, x, c)$ with $\tau(c) = c$ instead. The difference fades away when we focus on stable conjugacy classes.

These parametrizations above apply to all fields $F$ with $\mathrm{char}(F) \neq 2$.
	
Now fix $\mathbf{G}^! \in \Endo_{\elli}(\tilde{G})$ corresponding to $(n', n'')$. Let $\gamma = (\gamma', \gamma'') \in G^!_{\mathrm{sreg}}(F)$ (resp.\ $\delta \in G_{\mathrm{reg}}(F)$). In view of \cite[Corollaire 5.5]{Li11}, we say that $\gamma$ corresponds to $\delta$, written as $\gamma \leftrightarrow \delta$, if $\delta$ is parameterized by $(K, K^\natural, x, c)$ such that
\begin{itemize}
	\item there exists a decomposition $K = K' \times K''$, compatible with the involution $\tau$ so that $K^\natural = (K')^\natural \times (K'')^\natural$, and $x = (x', x'')$ accordingly;
	\item $\gamma'$ is parameterized by $(K', (K')^\natural, x', c')$ for some $c' \in (K')^{\times}$;
	\item $\gamma''$ is parameterized by $(K'', (K'')^\natural, -x'', c'')$ for some $c'' \in (K'')^{\times}$ (beware of the minus sign here).
\end{itemize}
This notion depends only on the stable conjugacy classes of $\gamma$ and $\delta$. If $\gamma \leftrightarrow \delta$ for some $\delta \in G_{\mathrm{reg}}(F)$, we say $\gamma$ is $G$-regular. The $G$-regular locus is Zariski open and dense in $G^!$.

Given $(K, K^\natural, x, c)$ as above, note that $\left\{x \in K^{\times}: x\tau(x) = 1 \right\}$ is the group of $F$-points of an $F$-torus, denoted by $K^1$ by abusing notations. If $\gamma \leftrightarrow \delta$, then $G^!_\gamma \simeq G_\delta$: indeed, both are isomorphic to $K^1 = (K')^1 \times (K'')^1$.

Note that the definition of $\Endo_{\elli}(\tilde{G}) = \Endo_{\elli}(\tilde{G}^{(2)})$ and the correspondence between conjugacy classes are insensitive to $\bpsi$.
		
\subsection{Transfer}\label{sec:transfer}
Fix $\bpsi$ and $(W, \lrangle{\cdot|\cdot})$. Write $\tilde{G}_{\mathrm{reg}} := \rev^{-1}\left( G_{\mathrm{reg}}(F) \right)$.

Let $\mathbf{G}^! \in \Endo_{\elli}(\tilde{G})$ correspond to $(n', n'')$. The transfer factor defined in \cite[\S 5.3]{Li11} is a map
\[ \Delta: G^!_{\mathrm{sreg}}(F) \times \tilde{G}_{\mathrm{reg}} \to \CC. \]
We record some of its basic properties below.
\begin{itemize}
	\item $\Delta(\gamma, z\tilde{\delta}) = z\Delta(\gamma, \tilde{\delta})$ for every $z \in \bmu_8$;
	\item $\Delta(\gamma, \tilde{\delta})$ depends only on the stable conjugacy class of $\gamma$ and the conjugacy class of $\tilde{\delta}$;
	\item putting $\delta = \rev(\tilde{\delta})$, we have $\Delta(\gamma, \tilde{\delta}) \neq 0$ only if $\gamma \leftrightarrow \delta$;
	\item if $\gamma \leftrightarrow \delta$ and $\tilde{\delta} \in \rev^{-1}(\delta)$, then $\Delta(\gamma, \tilde{\delta}) \in \bmu_8$.
\end{itemize}
	
A function $f: \tilde{G} \to \CC$ is said to be anti-genuine if	$f(z\tilde{x}) = z^{-1} f(\tilde{x})$ for all $\tilde{x} \in \tilde{G}$ and $z \in \bmu_8$. Denote the space of anti-genuine $C^\infty_c$-functions on $\tilde{G}$ as $C^\infty_{c, \asp}(\tilde{G})$, and observe that $G(F)$ acts by conjugation on $\tilde{G}$ and $\tilde{G}^{(2)}$.
\begin{itemize}
	\item For all $\tilde{\delta} \in \tilde{G}_{\mathrm{reg}}$ and $f \in C^\infty_{c, \asp}(\tilde{G})$, denote the corresponding normalized orbital integral as $I_{\tilde{G}}(\tilde{\delta}, f) := |D^G(\delta)|_F^{1/2} \int_{G_\delta(F) \backslash G(F)} f(g^{-1} \tilde{\delta} g) \dd g$, where $D^G(\delta)$ is the Weyl discriminant.
	\item For all $\gamma \in G^!_{\mathrm{sreg}}(F)$ and $f^! \in C^\infty_c(G^!(F))$, denote the corresponding normalized stable orbital integral as $S_{G^!}(\gamma, f^!)$.
\end{itemize}

Orbital integrals involve Haar measures: $I_{\tilde{G}}(\tilde{\delta}, f)$ (resp.\ $S_{G^!}(\gamma, f^!)$) is proportional to the Haar measure on $G(F)$ (resp.\ $G^!(F)$), and inverse proportional to the Haar measure on $G_\delta(F)$ (resp.\ $G^!_\gamma(F)$).

Define the spaces
\begin{align*}
	\orbI_{\asp}(\tilde{G}) & := C^\infty_{c, \asp}(\tilde{G}) \big/ \bigcap_{\tilde{\delta}} \Ker I_{\tilde{G}}(\tilde{\delta}, \cdot), \\
	S\orbI(G^!) & := C^\infty_c(G^!(F)) \big/ \bigcap_{\gamma} \Ker S_{G^!}(\gamma, \cdot);
\end{align*}
they carry natural structures of LF-spaces when $F$ is Archimedean. Normalized orbital integrals (resp.\ stable orbital integrals) can be evaluated on elements of $\orbI_{\asp}(\tilde{G})$ (resp.\ $S\orbI(G^!)$) once the Haar measures are chosen.

Let $\mes(G)$ (resp.\ $\mes(G^!)$) be the line spanned by Haar measures on $G(F)$ (resp.\ $G^!(F)$). The geometric transfer established in \cite{Li11} is a linear map
\begin{equation*}
	\Trans_{\mathbf{G}^!, \tilde{G}}: \orbI_{\asp}(\tilde{G}) \otimes \mes(G) \to S\orbI(G^!) \otimes \mes(G^!)
\end{equation*}
characterized as follows: $f^! = \Trans_{\mathbf{G}^!, \tilde{G}}(f)$ if and only if
\begin{equation*}
	 \sum_{\delta: \gamma \leftrightarrow \delta} \Delta(\gamma, \tilde{\delta}) I_{\tilde{G}}(\tilde{\delta}, f) = S_{G^!}(\gamma, f^!)
\end{equation*}
for all $G$-regular $\gamma$, where $\tilde{\delta}$ is an arbitrary element of $\rev^{-1}(\delta)$. The dependence of $I_{\tilde{G}}(\tilde{\delta}, \cdot)$ (resp.\ $S_{G^!}(\gamma, \cdot)$) on the Haar measure of $G(F)$ (resp.\ $G^!(F)$) is absorbed into $\mes(G)$ (resp.\ $\mes(G^!)$). On the other hand, we use compatible Haar measures on $G^!_\gamma(F)$ and $G_\delta(F)$ whenever $\gamma \leftrightarrow \delta$ (so that $G^!_\gamma \simeq G_\delta$). Therefore, $\Trans_{\mathbf{G}^!, \tilde{G}}$ involves no choice of Haar measures.

Let $D_-(\tilde{G})$ (resp.\ $SD(G^!)$) denote the linear dual of $\orbI_{\asp}(\tilde{G})$ (resp.\ $S\orbI(G^!)$), continuous for Archimedean $F$; for example, the characters (resp.\ stable characters) attached to genuine tempered irreducible representations (resp.\ to bounded L-parameters) are elements thereof. The dual of $\Trans_{\mathbf{G}^!, \tilde{G}}$ is denoted by
\[ \trans_{\mathbf{G}^!, \tilde{G}}: SD(G^!) \otimes \mes(G^!)^\vee \to D_-(\tilde{G}) \otimes \mes(G)^\vee. \]

In \cite{Li19}, it is shown that $\trans_{\mathbf{G}^!, \tilde{G}}$ sends stable virtual characters on $G^!(F)$ to genuine virtual characters on $\tilde{G}$. The map $\trans_{\mathbf{G}^!, \tilde{G}}$ is called the spectral transfer.

\section{Review of local Langlands correspondence}\label{sec:review-LLC}
\subsection{The correspondence}\label{sec:LLC}
Let $\Pi_-(\tilde{G})$ be the set of isomorphism classes of irreducible genuine representations of $\tilde{G}$. The representations are understood as smooth ones if $F$ is non-Archimedean, and as Harish-Chandra modules if $F$ is Archimedean by fixing a maximal compact subgroup.

Let $\Pi_{\mathrm{temp}, -}(\tilde{G})$ be the subset of $\Pi_-(\tilde{G})$ consisting of tempered irreducible genuine representations. For every connected reductive $F$-group $R$, define $\Pi_{\mathrm{temp}}(R) \subset \Pi(R)$ in the similar way.

On the other hand, denote by $\Phi_{\mathrm{bdd}}(\tilde{G})$ the set of equivalence classes of bounded L-parameters for $\tilde{G}$, using the definition of $\tilde{G}^\vee$ in \S\ref{sec:endoscopic-data}. For each $\phi \in \Phi_{\mathrm{bdd}}(\tilde{G})$, we have the groups
\[ S_\phi := Z_{\tilde{G}^\vee}(\Image(\phi)), \quad \EuScript{S}_\phi := \pi_0(S_\phi). \]
Then $\EuScript{S}_\phi$ is isomorphic to a finite power of $\bmu_2$; denote its Pontryagin dual by $\EuScript{S}_\phi^\vee$.

The local Langlands correspondence (LLC) for $\tilde{G}$ in the tempered setting is a decomposition
\begin{equation}\label{eqn:LLC-Mp}
	\begin{aligned}
		\Pi_{\mathrm{temp}, -}(\tilde{G}) & = \bigsqcup_{\phi \in \Phi_{\mathrm{bdd}}(\tilde{G})} \Pi_\phi, \\
		\EuScript{S}_\phi^\vee & \xleftrightarrow{1:1} \Pi_\phi \\
		\chi & \longmapsto \pi_{\phi, \chi}.
	\end{aligned}
\end{equation}
The definitions of $\Pi_\phi$ and $\pi_{\phi, \chi}$ are canonical once $\bpsi$ and $(W, \lrangle{\cdot|\cdot})$ are chosen. A truly canonical formulation requires the L-group of $\tilde{G}$, cf.\ \S\ref{sec:remark-L}.

Moreover, if we denote by $\Pi_{2, -}(\tilde{G})$ (resp.\ $\Phi_{2, \mathrm{bdd}}(\tilde{G})$) the subset of square-integrable representations (resp.\ discrete series L-parameters), then the \eqref{eqn:LLC-Mp} restricts to
\[ \Pi_{2, -}(\tilde{G}) = \bigsqcup_{\phi \in \Phi_{2, \mathrm{bdd}}(\tilde{G})} \Pi_\phi. \]

These properties also hold for the LLC established in \cite{Ar13} for quasisplit classical groups, such as $\SO(2n+1)$.

The correspondence \eqref{eqn:LLC-Mp} is established in \cite{AB98} for $F = \R$, and \cite{GS1} for non-Archimedean $F$. In these cases, it is actually given by $\Theta$-lifting: for each quadratic $F$-vector space $(V, q)$ of dimension $2n+1$ and discriminant $1$, it is proved in \textit{loc.\ cit.}\ that for every $\sigma \in \Pi_{\mathrm{temp}}(\SO(V))$, there exists a unique extension of $\sigma$ to $\Or(V)$ whose $\Theta$-lift to $\tilde{G}$ is non-zero, the genuine representation so obtained is tempered irreducible, and this procedure yields
\begin{equation}\label{eqn:LLC-Mp-Theta}
	\Pi_{\mathrm{temp}, -}(\tilde{G}) \xleftrightarrow{1:1} \bigsqcup_{\substack{\dim V = 2n+1 \\ \mathrm{disc}(V) = 1 \\ \text{up to isom.}}} \Pi_{\mathrm{temp}}(\SO(V)).
\end{equation}

On the other hand, the LLC for special odd orthogonal groups is known by \cite{Ar13}: it decomposes the right hand side of \eqref{eqn:LLC-Mp-Theta} into Vogan L-packets indexed by L-parameters $\mathcal{L}_F \to \Sp(2n, \CC)$. This is how \eqref{eqn:LLC-Mp} arises.

We now turn to the easier case $F = \CC$. Since $\tilde{G}$ splits uniquely, $\Pi_{\mathrm{temp}, -}(\tilde{G}) = \Pi_{\mathrm{temp}}(G)$. Fix a symplectic basis of $W$ to obtain the standard Borel pair $(B, T)$ of $G$. We use the following fact: let $\mathrm{W}(G, T)$ be the Weyl group and $\eta$ be any unitary character of $T(F)$, then the normalized parabolic induction of $\eta$ to $G(F)$ is irreducible and tempered, and this induces a bijection
\[ \left\{\eta: T(F) \to \CC^{\times}, \;\text{unitary} \right\} \big/ \mathrm{W}(G, T) \xleftrightarrow{1:1} \Pi_{\mathrm{temp}}(G) . \]
All these results are due to Zhelobenko; we refer to \cite[\S\S 2.2--2.4]{MR17} for a summary.

Note that $T$ also embeds into $\SO(2n+1)$ as a maximal torus, with the same Weyl group. The left hand side above is then in bijection with $\Phi_{\mathrm{bdd}}(\SO(2n+1)) = \Phi_{\mathrm{bdd}}(\tilde{G})$ (noting that $\mathcal{L}_{\CC} = \CC^{\times}$). This gives \eqref{eqn:LLC-Mp} for $F = \CC$.

The recipe \eqref{eqn:LLC-Mp-Theta} via $\Theta$-lifting also applies when $F = \CC$. Indeed, there is only one $(V, q)$, and one concludes from the trivial case $n = 0$ together with the induction principle \cite[Corollary 3.21]{AB95} for reductive dual pairs of type I over $\CC$.

Finally, the LLC extends to all genuine irreducible representations of $\tilde{G}$ by passing to Langlands quotients. This step is straightforward, and will not be considered in this article.

\subsection{Endoscopic character relation}
We refer to \cite[\S 5]{GGP1} for an overview of $\epsilon$-factors. Let $\phi \in \Phi_{\mathrm{bdd}}(\tilde{G})$ and set $S_{\phi, 2} := \{s \in S_\phi: s^2 = 1 \}$. View $\phi$ as a $2n$-dimensional representation of $\mathcal{L}_F$, self-dual of symplectic type, with a commuting action of $S_\phi$. Denote the $(-1)$-eigenspace of $s$ as $\phi^{s = -1}$ and set
\[ \epsilon(\phi^{s = -1}) := \epsilon\left(\frac{1}{2}, \phi^{s = -1}, \bpsi \right). \]
This is $\bmu_2$-valued and independent of $\bpsi$, since $\phi^{s = -1}$ is also self-dual of symplectic type.

Given $\phi \in \Phi_{\mathrm{bdd}}(\tilde{G})$ and $s \in S_{\phi, 2}$, define the genuine virtual character
\begin{equation}\label{eqn:T}
	T_{\phi, s} := \epsilon(\phi^{s = -1}) \trans_{\mathbf{G}^!, \tilde{G}} \left( S\Theta^{G^!}_{\phi^!} \right)
\end{equation}
on $\tilde{G}$, where
\begin{itemize}
	\item $\mathbf{G}^! \in \Endo_{\elli}(\tilde{G})$ is determined by the conjugacy class of $s$;
	\item $\phi$ factors through $Z_{\tilde{G}^\vee}(s) = (G^!)^\vee$, and gives rise to $\phi^! \in \Phi_{\mathrm{bdd}}(G^!)$;
	\item $S\Theta^{G^!}_{\phi^!}$ is the stable tempered character on $G^!(F)$ attached to $\phi^!$ by Arthur's theory \cite{Ar13}, which is independent of Whittaker data.
\end{itemize}

A priori, $T_{\phi, s}$ depends on $\phi$ and the conjugacy class of $s$ in $S_\phi$. However we have the following property.

\begin{lemma}[special case of {\cite[Lemma 4.3.3]{Li24a}}]\label{prop:T}
	Given $\phi$, the distribution $T_{\phi, s}$ depends only on the image of $s \in S_{\phi, 2}$ in $\EuScript{S}_\phi$.
\end{lemma}

Note that $S_\phi$ is a direct product of complex orthogonal groups, general linear groups and symplectic groups --- see \eqref{eqn:S-phi-decomp} below. Hence $S_{\phi, 2}$ surjects onto $\EuScript{S}_\phi$. We obtain a distribution-valued map $x \mapsto T_{\phi, x}$ on $\EuScript{S}_\phi$. Its Fourier coefficients are exactly given by the character distributions $\Tr\left( \pi_{\phi, \chi} \right)$ of $\pi_{\phi, \chi}$, stated as follows.

\begin{theorem}[C.\ Luo \cite{Luo20}]\label{prop:Luo-endo}
	Let $\phi \in \Phi_{\mathrm{bdd}}(\tilde{G})$. The representations $\pi_{\phi, \chi}$ (where $\chi \in \EuScript{S}_\phi^\vee$) are characterized by the following identities: for every $x \in \EuScript{S}_\phi$ we have
	\[ T_{\phi, x} = \sum_{\chi \in \EuScript{S}_\phi^\vee} \chi(x) \Tr\left( \pi_{\phi, \chi} \right). \]
\end{theorem}

This gives another characterization of the LLC of $\tilde{G}$ that is based on endoscopic transfer.

\subsection{A theorem of Gan--Savin}
Let $\phi \in \Phi_{\mathrm{bdd}}(\tilde{G})$. As a representation of $\mathcal{L}_F$, it decomposes uniquely into
\[ \phi = \bigoplus_{i \in I} m_i \phi_i \]
where $\phi_i$ are simple and distinct, $m_i \in \Z_{\geq 1}$. The indexing set $I$ admits a decomposition
\[ I = I^+ \sqcup I^- \sqcup J \sqcup J', \]
with the following properties:
\begin{itemize}
	\item there is a bijection between $J$ and $J'$, written as $j \leftrightarrow j'$.
	\item $\phi_i$ is self-dual of symplectic (resp.\ orthogonal) type when $i \in I^+$ (resp.\ $i \in I^-$);
	\item $\phi_j$ is not self-dual when $j \in J$, and $\phi_{j'}$ is isomorphic to the contragredient of $\phi_j$;
	\item $m_i$ is even for all $i \in I^-$, and $m_j = m_{j'}$ for all $j \in J$.
\end{itemize}
From this we obtain
\begin{equation}\label{eqn:S-phi-decomp}\begin{aligned}
	S_\phi & \simeq \prod_{i \in I^+} \Or(m_i, \CC) \times \prod_{i \in I^-} \Sp(m_i, \CC) \times \prod_{j \in J} \GL(m_j, \CC), \\
	\EuScript{S}_\phi & \simeq \bmu_2^{I^+} .
\end{aligned}\end{equation}

Now fix $c \in F^{\times}$. The coset $c F^{\times 2}$ corresponds to a homomorphism $\zeta = \zeta_c: \Weil{F} \to \bmu_2 \simeq Z_{\tilde{G}^\vee}$, which can be used to twist L-parameters of $\tilde{G}$. Write the twisting action as $\phi \mapsto \phi\zeta$, and identify $\EuScript{S}_\phi^\vee = \EuScript{S}_{\phi\zeta}^\vee$ with $\bmu_2^{I^+}$ in what follows.

\begin{definition}\label{def:delta-c}
	Given $c F^{\times 2}$, let $\delta_c \in \EuScript{S}_\phi^\vee$ be given by
	\[ \delta_{c, i} = \zeta_c(-1)^{\frac{1}{2} \dim \phi_i} \dfrac{\epsilon(\frac{1}{2}, \phi_i, \bpsi)}{\epsilon(\frac{1}{2}, \phi_i \zeta_c, \bpsi)}. \]
	for all $i \in I^+$. This is canonically defined.
\end{definition}

Hereafter, we consider genuine irreducible representations of $\tilde{G}^{(2)}$. As recalled earlier, $\tilde{G}^{(2)}$ does not depend on $\bpsi$; on the other hand, the LLC \S\ref{sec:LLC} does so. To emphasize this dependence, we denote the genuine representation of $\tilde{G}^{(2)}$ attached to $(\phi, \chi)$ as $\pi_{\phi, \chi}^{\bpsi}$.

All additive characters of $F$ are obtained from $\bpsi$ by rescaling. The following result gives an exact description of the dependence of local Langlands correspondence on $\bpsi$.

\begin{theorem}[Gan--Savin {\cite[Theorem 12.1]{GS1}}]\label{prop:GS}
	Let $c \in F^{\times}$ and put $\zeta = \zeta_c$. For all $\phi \in \Phi_{\mathrm{bdd}}(\tilde{G})$ and $\chi \in \EuScript{S}_\phi^\vee$, we have
	\[ \pi^{\bpsi_c}_{\phi\zeta, \chi \delta_c} \simeq \pi^{\bpsi}_{\phi, \chi}. \]
\end{theorem}

This assertion appeared first in \cite[Conjecture 11.3]{GGP1}. Strictly speaking, it is only settled for non-Archimedean $F$ in \cite[Theorem 12.1]{GS1}. For $F = \R$, one may try to extract this from \cite{AB98}. A direct, endoscopic proof of Theorem \ref{prop:GS} will be given in \S\ref{sec:proof-GS}.

\section{Spinor norms}\label{sec:SN}
\subsection{Basic properties}\label{sec:SN-basic}
To begin with, let $F$ be any field with $\mathrm{char}(F) \neq 2$. For a quadratic $F$-vector space $(V, q)$, the spinor norm is a canonical homomorphism
\[ \mathrm{SN} = \mathrm{SN}_V: \Or(V) \to F^{\times} / F^{\times 2} \]
with the following properties; we refer to \cite[Chapter 9, \S 3]{Sc85} or \cite[\S 5.1]{FKS19} for details.
\begin{itemize}
	\item If $(V, q) = (V_1, q_1) \oplus (V_2, q_2)$, then $\mathrm{SN}_V|_{\Or(V_1) \times \Or(V_2)}$ is the product of $\mathrm{SN}_{V_1}$ and $\mathrm{SN}_{V_2}$.
	\item If $V$ is the direct sum of $n$ copies of the hyperbolic plane, so that $\GL(n) \hookrightarrow \SO(V)$, then
	\[ \mathrm{SN}|_{\GL(n)} = \det \bmod\; F^{\times 2}. \]
	\item The restriction of $\mathrm{SN}_V$ to $\SO(V)$ is invariant under dilation $q \mapsto tq$ where $t \in F^{\times}$. Indeed, its effect is to multiply $\mathrm{SN}(\tau)$ by $t F^{\times 2}$ for each reflection $\tau \in \Or(V)$.
	\item We have the following cohomological interpretation: $\mathrm{SN}_V: \SO(V) \to F^{\times}/F^{\times 2}$ equals the connecting homomorphism induced by the short exact sequence $1 \to \mu_2 \to \mathrm{Spin}(V) \to \SO(V) \to 1$ of group schemes over $F$.
\end{itemize}

Hereafter, we focus on the special case of $H := \SO(2n+1)$ over a local field $F$ with $\mathrm{char}(F) = 0$, although certain results can surely be generalized.

\begin{lemma}\label{prop:SN-formula}
	Let $\gamma \in H_{\mathrm{sreg}}(F)$, whose conjugacy class is parametrized by $(K, K^{\natural}, x, c)$ as in \S\ref{sec:conjugacy-class}. There exists $\omega \in K^{\times}$ such that $x = \omega/\tau(\omega)$, and for any such $\omega$ we have
	\[ \mathrm{SN}(\gamma) = N_{K|F}(\omega) \bmod F^{\times 2}. \]
	
	As a consequence, $\mathrm{SN}(\gamma)$ depends only on the stable conjugacy class of $\gamma$.
\end{lemma}
\begin{proof}
	Decompose $K = \prod_{i \in I} K_i$ and $K^{\natural} = \prod_{i \in I} K^{\natural}_i$ accordingly, so that $K^{\natural}_i$ is a field and $K_i$ is either a quadratic field extension of $K_i^{\natural}$, or $K_i \simeq K_i^{\natural} \times K_i^{\natural}$, for each $i \in I$. Write $x = (x_i)_i$.
	
	The existence of $\omega = (\omega_i)_i$ follows by Hilbert's Theorem 90 applied to each $K_i | K^{\natural}_i$ and $x_i$; it suffices to deal with those $i$ such that $K_i$ is a field.
	
	By inspecting the parametrization of conjugacy classes (see \cite[\S 3]{Li11} or \cite[\S 3.1]{Li20}), any quadratic $F$-vector space defining $H$ is seen to be the direct sum over $i \in I$ of the spaces considered in \cite[Fact 5.1.8]{FKS19} made from the data $(K_i, K^{\natural}_i, c_i)$, plus an anisotropic line, on which $\gamma$ acts by multiplication by $x_i$ and $\identity$, respectively. We can now apply \cite[Fact 5.1.8]{FKS19} to each summand to infer that
	\[ \mathrm{SN}(\gamma) = 1 \cdot \prod_{i \in I} N_{K_i|F}(\omega_i) = N_{K|F}(\omega) \bmod F^{\times 2}. \]
	
	The formula above does not involve $c$, hence  $\mathrm{SN}(\gamma)$ depends only on the stable conjugacy class of $\gamma$.
\end{proof}

\begin{remark}
	The stable invariance is a general fact for characters of $H(F)$ arising from $\Hm^1(\Weil{F}, Z_{H^\vee})$, which is indeed the case for $(t, \mathrm{SN})_{F, 2}$, for all $t \in F^{\times}$. From this we can also deduce the second assertion above.
\end{remark}

Consider a quadratic character $\zeta: \Weil{F} \to \bmu_2$, corresponding to a coset $cF^{\times 2}$ in $F^{\times}$. We may view $\zeta$ as valued in $Z_{H^\vee}$. Such homomorphisms give rise to quadratic characters $s_c: H(F) \to \bmu_2$ by \cite[\S 10.2]{Bo79}. They are related to spinor norms as follows.

\begin{lemma}\label{prop:coho-SN}
	Given a coset $cF^{\times 2}$ in $F^{\times}$, we have $s_c = (c, \mathrm{SN}(\cdot))_{F, 2}$.
\end{lemma}
\begin{proof}
	Immediate from the cohomological interpretation of $\mathrm{SN}$.
\end{proof}

Since $\mathrm{SN}$ is stably invariant on $H_{\mathrm{sreg}}(F)$, so is $s_c$. Hence the involution $f \mapsto s_c f$ of $C^\infty_c(H(F))$ descends to $S\orbI(H)$, and is continuous when $F$ is Archimedean. We tensor it with $\mes(H)$, and denote the resulting involution and its dual as
\begin{equation}\label{eqn:Upsilon}
	\begin{aligned}
		\Upsilon^{\zeta, H}: S\orbI(H) \otimes \mes(H) & \rightiso S\orbI(H) \otimes \mes(H), \\
		\Upsilon_\zeta^H: SD(H) \otimes \mes(H)^\vee & \rightiso SD(H) \otimes \mes(H)^\vee.
	\end{aligned} 
\end{equation}

\subsection{Effect on L-parameters}
Set $H = \SO(2n+1)$ as before. Fix a quadratic character $\zeta: \Weil{F} \to \bmu_2$, corresponding to a coset $cF^{\times 2}$ in $F^{\times}$.

For each $\phi \in \Phi_{\mathrm{bdd}}(H)$, the stable tempered character $S\Theta^H_\phi$ belongs to $SD(H) \otimes \mes(H)^\vee$. Viewing $\zeta$ as a homomorphism valued in $Z_{H^\vee}$, we obtain an involution $\phi \mapsto \zeta\phi$ of $\Phi_{\mathrm{bdd}}(H)$, and similarly for L-parameters of the Levi subgroups of $H$.

We are going to determine the effect of $\phi \mapsto \zeta\phi$ on the LLC for $H$ in terms of spinor norms. Given Lemma \ref{prop:coho-SN}, this ought to be a standard property of local Langlands correspondence. Due to the lack of adequate references, we give a direct proof below.

\begin{lemma}\label{prop:zeta-packet-prep}
	Let $\phi \in \Phi_{2, \mathrm{bdd}}(H)$, and $\phi^{\GL}$ be a bounded L-parameter for $\GL(2n)$ that is multiplicity-free and each simple summand is self-dual of symplectic type. Consider
	\begin{itemize}
		\item $\sigma \in \Pi^H_\phi$,
		\item $\sigma^{\GL}$: the tempered irreducible representation of $\GL(2n, F)$ parametrized by $\phi^{\GL}$.
	\end{itemize}
	Embed $\GL(2n) \times H$ as a Levi subgroup of $L := \SO(6n+1)$, then the following are equivalent:
	\begin{enumerate}[(i)]
		\item the normalized parabolic induction of $\sigma^{\GL} \boxtimes \sigma$ to $L(F)$ is irreducible;
		\item $\phi$ maps to $\phi^{\GL}$ under $\Phi_{\mathrm{bdd}}(H) \hookrightarrow \Phi_{\mathrm{bdd}}(\GL(2n))$.
	\end{enumerate}
\end{lemma}
\begin{proof}
	Decompose $\phi^{\GL}$ into simple summands
	\[ \phi^{\GL} = \phi_1^{\GL} \oplus \cdots \oplus \phi_r^{\GL}, \quad n_i := \dim \phi_i. \]
	Set $M := \prod_{i=1}^r \GL(n_i) \times H$, viewed as a Levi subgroup of $L$, then $\phi_M := (\phi_1, \ldots, \phi_r, \phi) \in \Phi_{2, \mathrm{bdd}}(M)$; its image in $\Phi_{\mathrm{bdd}}(L)$ is $\phi_L = 2\phi^{\GL} \oplus \phi$. There is then a natural homomorphism $S_{\phi_M} \to S_{\phi_L}$.
	
	Consider Arthur's $R$-group $R_{\phi_M}$ in \cite[\S 2.4]{Ar13}, defined relative to $M \subset L$. It is canonically isomorphic to $\EuScript{S}_{\phi_L} / \EuScript{S}_{\phi_M}$ by \textit{loc.\ cit.} From this one readily sees
	\[ R_{\phi_M} = \{1\} \iff \phi \mapsto \phi^{\GL}. \]
	
	On the other hand, $\sigma^{\GL} = \sigma_1 \times \cdots \times \sigma_r$ where $\sigma_i$ is the irreducible representation of $\GL(n_i, F)$, square-integrable modulo center, parametrized by $\phi_i$. Given $\sigma \in \Pi^H_\phi$, put $\sigma_M = \sigma_1 \boxtimes \cdots \boxtimes \sigma_r \boxtimes \sigma$. By \cite[(6.5.3)]{Ar13}, the Knapp--Stein $R$-group $R_{\sigma_M}$ defined relative to $M \subset L$ satisfies
	\[ R_{\sigma_M} \simeq R_{\phi_M}. \]
	
	All in all, the parabolic induction of $\sigma^{\GL} \boxtimes \sigma$ from $\GL(2n) \times H$ to $L$ is irreducible if and only if $R_{\phi_M} \simeq R_{\sigma_M} = \{1\}$, if and only if $\phi$ maps to $\phi^{\GL}$.
\end{proof}

\begin{proposition}\label{prop:zeta-packet}
	Let $\phi \in \Phi_{\mathrm{bdd}}(H)$, we have
	\[ \Pi^H_{\zeta\phi} = \left\{ (c, \mathrm{SN})_{F, 2} \otimes \sigma : \sigma \in \Pi^H_\phi \right\}. \]
\end{proposition}
\begin{proof}
	First off, suppose $\phi \in \Phi_{2, \mathrm{bdd}}(H)$. It suffices to show
	\[ \sigma \in \Pi^H_\phi \implies (c, \mathrm{SN})_{F, 2} \otimes \sigma \in \Pi^H_{\zeta\phi}. \]
	
	In Lemma \ref{prop:zeta-packet-prep}, take $\phi^{\GL}$ to be the image of $\phi$ and any $\sigma \in \Pi^H_\phi$, so that $\sigma^{\GL} \boxtimes \sigma$ induces irreducibly to $L(F)$. By \S\ref{sec:SN-basic}, the spinor norm on $L(F)$ restricts to the product of $\det \bmod\; F^{\times 2}$ and $\mathrm{SN}$ on $\GL(2n, F) \times H(F)$. Tensoring by the character $(c, \mathrm{SN})_{F, 2}$ of $L(F)$ does not affect the irreducibility of parabolic induction, thus
	\[ \left( (c, \det)_{F, 2} \otimes \sigma^{\GL}\right) \boxtimes \left( (c, \mathrm{SN})_{F, 2} \otimes \sigma \right) \;\text{induces irreducibly to}\; L(F). \]
	
	Note that $\zeta\phi^{\GL}$ is the image of $\zeta\phi$, and the corresponding representation of $\GL(2n, F)$ is $(c, \det)_{F, 2} \otimes \sigma^{\GL}$. On the other hand, $\sigma' := (c, \mathrm{SN})_{F, 2} \otimes \sigma$ is still square-integrable modulo center. Denote the L-parameter of $\sigma'$ by $\phi' \in \Phi_{2, \mathrm{bdd}}(H)$. Lemma \ref{prop:zeta-packet-prep} implies $\phi' \mapsto \zeta\phi^{\GL}$, i.e.\ $\sigma' \in \Pi^H_{\zeta\phi}$. This concludes the case $\phi \in \Phi_{2, \mathrm{bdd}}(H)$.
	
	Consider a general $\phi \in \Phi_{\mathrm{bdd}}(H)$. It is the image of $\phi_0 \in \Phi_{2, \mathrm{bdd}}(M_H)$ for some Levi subgroup $M_H \subset H$, and $\Pi^H_\phi$ consists of all irreducible constituents of parabolic inductions from various $\sigma_0 \in \Pi^{M_H}_{\phi_0}$, by \cite[\S 2.4]{Ar13}.
	
	If $\phi$ is replaced by $\zeta\phi$, then $\phi_0$ is replaced by $\zeta\phi_0$ and all $\sigma_0$ are tensored by the character $(c, \det)_{F, 2}$ (resp.\ $(c, \mathrm{SN})_{F, 2}$) on the $\GL$ factors (resp.\ $\SO$ factor) by the previous step. We have seen that this character of $M_H(F)$ is the restriction of $(c, \mathrm{SN})_{F, 2}$ on $H(F)$. Hence $\Pi^H_{\zeta\phi}$ is obtained from $\Pi^H_\phi$ by tensoring $(c, \mathrm{SN})_{F, 2}$.
\end{proof}

\begin{corollary}\label{prop:Upsilon-effect}
	Consider a quadratic character $\zeta: \Weil{F} \to \bmu_2$. The $\Upsilon_\zeta^H$ in \eqref{eqn:Upsilon} satisfies $\Upsilon_\zeta^H \left( S\Theta^H_\phi \right) = S\Theta^H_{\zeta\phi}$.
\end{corollary}
\begin{proof}
	In view of Lemma \ref{prop:coho-SN} and Proposition \ref{prop:zeta-packet}, it suffices to recall that
	\[ S\Theta^H_\phi = \sum_{\sigma \in \Pi^H_\phi} \Tr(\sigma) \]
	where $\Tr(\sigma)$ is the character distribution of $\sigma$, and same for $S\Theta^H_{\zeta\phi}$.
\end{proof}

Hence $\Upsilon_\zeta^H$ agrees with the eponymous involution defined in \cite[\S 9.1]{Li24a}, when restricted to the subspace generated by stable tempered characters.
	
\section{Variation of additive characters}\label{sec:var}
\subsection{Action of similitude groups}
Let $\nu: \GSp(W) \to \Gm$ be the similitude character. Giving $g \in \GSp(W)$ with $c := \nu(g)$ is the same as giving an isomorphism of symplectic $F$-vector spaces $g: (W, c\lrangle{\cdot|\cdot}) \rightiso (W, \lrangle{\cdot|\cdot})$. Hence $\Ad(g)$ induces an automorphism
\[ G = \Sp(W, c\lrangle{\cdot|\cdot}) \rightiso \Sp(W, \lrangle{\cdot|\cdot}) = G . \]

\begin{itemize}
	\item On the level of twofold coverings, $\Ad(g)$ lifts uniquely to an isomorphism $\tilde{G}^{(2)} \rightiso \tilde{G}^{(2)}$, and induces $\identity$ on $\bmu_2$. This follows from the classification \cite[Theorem 10.4]{Mo68} of coverings for $G(F)$.
	
	\item On the level of eightfold coverings, let us denote $\tilde{G}$ as $\tilde{G}^{\bpsi, \lrangle{\cdot|\cdot}}$ to emphasize its dependence on these data. By a transport of structure in the construction of Weil representations and metaplectic groups (namely Schrödinger models), we see $\Ad(g)$ lifts canonically to
	\[ \tilde{G}^{\bpsi, c\lrangle{\cdot|\cdot}} \rightiso \tilde{G}^{\bpsi, \lrangle{\cdot|\cdot}}, \; \text{inducing $\identity$ on $\bmu_8$.} \]
	A closer inspection shows $\tilde{G}^{\bpsi, \lrangle{\cdot|\cdot}}$ depends only on $\bpsi \circ \lrangle{\cdot|\cdot}$, hence
	\[ \tilde{G}^{\bpsi, c\lrangle{\cdot|\cdot}} = \tilde{G}^{\bpsi_c, \lrangle{\cdot|\cdot}} =: \tilde{G}^{\bpsi_c}. \]
\end{itemize}

As in \S\ref{sec:endoscopic-data}, we view the covering $\tilde{G}^{(2)} \twoheadrightarrow G(F)$ as an object independent of $\bpsi$; it embeds uniquely into $\tilde{G}^{\bpsi}$ as a sub-covering. Summing up, $g \in \GSp(W)$ gives rise to a canonical commutative diagram
\begin{equation}\label{eqn:Adg-lift}
	\begin{tikzcd}
		\tilde{G}^{\bpsi_c} \arrow[r, "\sim"] \arrow[dd, bend right=45, "{\rev}"'] & \tilde{G}^{\bpsi} \arrow[dd, bend left=45, "{\rev}"] \\
		\tilde{G}^{(2)} \arrow[hookrightarrow, u] \arrow[r, "\sim"] \arrow[d, "{\rev^{(2)}}"] & \tilde{G}^{(2)} \arrow[hookrightarrow, u] \arrow[d, "{\rev^{(2)}}"'] \\
		G(F) \arrow[r, "\sim"] & G(F)
	\end{tikzcd}
\end{equation}
where all the horizontal isomorphisms will be denoted as $\Ad(g)$, and $\Ad(g_1 g_2) = \Ad(g_1) \Ad(g_2)$ continues to hold on the level of coverings.

Denote the transfer factor $\Delta$ (see \S\ref{sec:transfer}) by $\Delta^{\bpsi}$ or $\Delta^{\bpsi, \lrangle{\cdot|\cdot}}$ to emphasize its dependence on these data. However, $\mathbf{G}^! \in \Endo_{\elli}(\tilde{G})$ and the correspondence of stable conjugacy classes (see \S\ref{sec:conjugacy-class}) do not depend on them.

\begin{lemma}\label{prop:Ad-transport}
	Let $g \in \GSp(W)$ and $c := \nu(g)$. If $\gamma \in G^!_{\mathrm{sreg}}(F)$ corresponds to $\delta \in G_{\mathrm{reg}}(F)$, and $\tilde{\delta} \in \tilde{G}^{\bpsi_c} = \tilde{G}^{\bpsi, c\lrangle{\cdot|\cdot}}$ maps to $\delta$, then
	\[ \Delta^{\bpsi_c}(\gamma, \tilde{\delta}) = \Delta^{\bpsi, c\lrangle{\cdot|\cdot}}(\gamma, \tilde{\delta}) = \Delta^{\bpsi, \lrangle{\cdot|\cdot}}(\gamma, \Ad(g)(\tilde{\delta})). \]
\end{lemma}
\begin{proof}
	By inspecting the definition of $\Delta$ in \cite[\S 5.3]{Li11}, especially the part concerning the character of Weil representations $\omega_{\bpsi}^{\pm}$, we see that $\Delta$ depends only on $\bpsi \circ \lrangle{\cdot|\cdot}$. The first equality follows.
	
	Similarly, in view of \textit{loc.\ cit.}, the second equality is simply a transport of structure by $g: (W, c\lrangle{\cdot|\cdot}) \rightiso (W, \lrangle{\cdot|\cdot})$.
\end{proof}

\subsection{Calibration}\label{sec:calibration}
Consider $g \in \GSp(W)$ and $c := \nu(g)$. When $\tilde{\delta}$ is given, the conjugacy class of $\Ad(g)(\tilde{\delta})$ depends only on $c$.

\begin{lemma}\label{prop:st-conj}
	Let $\tilde{\delta} \in \tilde{G}^{(2)}$ with $\delta := \rev^{(2)}(\tilde{\delta}) \in G_{\mathrm{reg}}(F)$. Parametrize the conjugacy class of $\delta$ by the datum $(K, K^{\natural}, x, c)$ as in \S\ref{sec:conjugacy-class}, then
	\begin{enumerate}[(i)]
		\item there exists $\omega \in K^{\times}$ such that $\omega/\tau(\omega) = x$, and
		\[ \mathrm{CAd}(g)(\tilde{\delta}) := \left(c, N_{K|F}(\omega)\right)_{F, 2} \Ad(g)(\tilde{\delta}) \; \in \tilde{G}^{(2)} \]
		is well-defined;
		\item $\mathrm{CAd}(g)(\tilde{\delta})$ is stably conjugate to $\tilde{\delta}$ in the sense of Adams.
	\end{enumerate}
	We refer to \cite[\S 9.1]{Li20} for a review of Adams' notion of stable conjugacy in $\tilde{G}^{(2)}$.
\end{lemma}
\begin{proof}
	For (i), the existence of $\omega$ has been explained in Lemma \ref{prop:SN-formula}. The choice of $\omega$ is unique up to $(K^{\natural})^{\times}$, hence $N_{K|F}(\omega)$ is unique up to $F^{\times 2}$, and $\mathrm{CAd}(g)(\tilde{\delta})$ is well-defined.
	
	Let $S := G_\delta$. In \cite[Definition--Proposition 4.3.7]{Li20} (with $m=2$ and noting that $\iota_{Q, 2}$ is an isomorphism), one defines
	\begin{itemize}
		\item a group $G^S$ such that $S \subset G^S \subset G$, which is a product of $\SL(2)$ over various finite extensions of $F$,
		\item elements $\mathbf{C}\mathrm{Ad}(g_{\mathrm{ad}})(\tilde{\delta}) \in \tilde{G}^{(2)}$ for $g_{\mathrm{ad}} \in G^S_{\mathrm{ad}}(F)$, where $G^S_{\mathrm{ad}}$ is the adjoint group of $G^S$;
	\end{itemize}
	the latter item equals $\Ad(g_{\mathrm{ad}})(\tilde{\delta})$ times an explicit sign, exploiting the fact that conjugation by $G^S_{\mathrm{ad}}(F)$ makes sense here (see \textit{loc.\ cit.}) We claim that $\mathrm{CAd}(g)(\tilde{\delta})$ is $G(F)$-conjugate to $\mathbf{C}\mathrm{Ad}(g_{\mathrm{ad}})(\tilde{\delta})$ for some $g_{\mathrm{ad}} \in G^S_{\mathrm{ad}}(F)$.
	
	By a comparison of signs using \cite[Definition--Proposition 4.2.7]{Li20}, it suffices to show $\Ad(g_{\mathrm{ad}})(\tilde{\delta})$ is $G(F)$-conjugate to $\Ad(g)(\tilde{\delta})$ in $\tilde{G}^{(2)}$ for some $g_{\mathrm{ad}}$. Indeed, take
	\[ g_{\mathrm{ad}} = (g_i)_i \in G^S_{\mathrm{ad}}(F) \quad\text{such that}\quad \forall i, \; \det(g_{i, 1}) = c \]
	in the terminologies of \textit{loc.\ cit.}, then $\Ad(g_{\mathrm{ad}})$ is conjugation by some $g_1 \in \GSp(W)$ with $\nu(g_1) = c$, but $\Ad(g_1)(\tilde{\delta})$ and $\Ad(g)(\tilde{\delta})$ are conjugate.
	
	By the claim above and \cite[Theorem 9.2.3]{Li20}, we conclude that $\mathrm{CAd}(g)(\tilde{\delta})$ is stably conjugate to $\tilde{\delta}$ in $\tilde{G}^{(2)}$.
\end{proof}

The factor $(N_{K|F}(\omega), c)_{F, 2}$ above is called a calibration factor in \cite{Li20}.

Now consider $\mathbf{G}^! \in \Endo_{\elli}(\tilde{G}^{(2)})$ and suppose that $\gamma = (\gamma', \gamma'') \in G^!_{\mathrm{sreg}}(F)$ corresponds to $\delta \in G_{\mathrm{reg}}(F)$. There are decompositions $K = K' \times K''$ and $K^{\natural} = (K')^{\natural} \times (K'')^{\natural}$, etc.\, as reviewed in \S\ref{sec:conjugacy-class}.

We also have decompositions $K = \prod_{i \in I} K_i$ and $K^{\natural} = \prod_{i \in I} K^{\natural}_i$, so that each $K^{\natural}_i$ is a field, and $K_i$ is either a quadratic field extension of $K_i^{\natural}$ or $K_i^{\natural} \times K_i^{\natural}$, for each $i \in I$. The decompositions respect $K = K' \times K''$ and $K^{\natural} = (K')^{\natural} \times (K'')^{\natural}$; the indexing set $I$ decomposes accordingly into $I' \sqcup I''$ (see \cite[Corollaire 5.5]{Li11} for details).

For each $i \in I_0$, define $\sgn_{K_i|K^{\natural}_i}$ to be the quadratic character of $(K^{\natural}_i)^{\times}$ that is
\begin{itemize}
	\item associated with the extension $K_i|K^{\natural}_i$ if $K_i$ is a field;,
	\item trivial otherwise.
\end{itemize}
Define
\[ \sgn'' := \prod_{i \in I''} \sgn_{K_i | K^{\natural}_i}: (K'')^{\natural, \times} = \prod_{i \in I''} (K^{\natural}_i)^{\times} \to \bmu_2. \]

\begin{lemma}\label{prop:Delta-cocycle}
	For $\mathbf{G}^!$, $\gamma$ and $\delta$ given as above, and $\tilde{\delta} \in \tilde{G}^{(2)}$ such that $\rev^{(2)}(\tilde{\delta}) = \delta$, the transfer factor for $\mathbf{G}^!$ satisfies
	\[ \Delta(\gamma, \mathrm{CAd}(g)(\tilde{\delta})) = \sgn''(c) \Delta(\gamma, \tilde{\delta}) \]
	for all $g \in \GSp(W)$ with $c := \nu(g)$.
\end{lemma}
\begin{proof}
	Note that $\Ad(g)(\delta)$ is stably conjugate to $\delta$ in $G(F)$. Set $S := G_\delta$. The point is to describe the ``relative position''  between $\delta$ and $\Ad(g)(\delta)$, defined as a Galois cohomology class
	\[ \mathrm{inv}(\delta, \Ad(g)(\delta)) \in \Hm^1(F, S) \simeq K^{\natural, \times} / N_{K|K^{\natural}}(K^{\times}); \]
	it depends only on $c$, since so does the conjugacy class of $\Ad(g)(\delta)$.

	We use the computation in \cite[Proposition 3.3.4]{Li20}. In the cited result, one does not conjugate by $\GSp(W)$, but by some $(g_i)_i \in G^S_{\mathrm{ad}}(F)$; see the proof of Lemma \ref{prop:st-conj}. Take $(g_i)_i$ so that $\det(g_{i, 1}) = c$ for all $i$ in the terminology therein, then it amounts to conjugation by some $g_1 \in \GSp(W)$ with $\nu(g_1) = c$.
	
	Since $\Ad(g_1)(\delta)$ is conjugate to $\Ad(g)(\delta)$, the cited result says $\mathrm{inv}(\delta, \Ad(g)(\delta)) = c N_{K|K^{\natural}}(K^{\times})$. Since $\mathrm{CAd}(g)(\tilde{\delta})$ is stably conjugate to $\tilde{\delta}$ in $\tilde{G}^{(2)}$ by Lemma \ref{prop:st-conj}, we conclude by the cocycle property \cite[Proposition 5.13]{Li11} of $\Delta$.
\end{proof}

\subsection{Effect on transfer}
Let $c \in F^{\times}$. To $cF^{\times 2}$ is attached a quadratic character $\zeta: \Weil{F} \to \bmu_2 \simeq Z_{\tilde{G}^\vee}$. It also determines a quadratic character of $F^{\times}$, namely $(c, \cdot)_{F, 2}$.

For each $\mathbf{G}^! \in \Endo_{\elli}(\tilde{G}) = \Endo_{\elli}(\tilde{G}^{(2)})$ corresponding to $(n', n'')$, based on \eqref{eqn:Upsilon}, we define the involutions
\begin{align*}
	\Upsilon^\zeta & := \Upsilon^{\zeta, \SO(2n'+1)} \otimes \Upsilon^{\zeta, \SO(2n''+1)}, \\
	\Upsilon_\zeta & := \Upsilon_\zeta^{\SO(2n'+1)} \otimes \Upsilon_\zeta^{\SO(2n''+1)}
\end{align*}
for $S\orbI(G^!) \otimes \mes(G^!)$ and $SD(G^!) \otimes \mes(G^!)^\vee$, respectively; note that for Archimedean $F$, one should take nuclear $\widehat{\otimes}$ in the definitions above.

\begin{proposition}\label{prop:Delta-var}
	Given $\mathbf{G}^! \in \Endo_{\elli}(\tilde{G}^{(2)})$, define the character $s^!_c$ of $G^!(F)$ by
	\[ s^!_c(\gamma) = (c, \mathrm{SN}(\gamma'))_{F, 2} (c, \mathrm{SN}(\gamma''))_{F, 2}, \quad \gamma = (\gamma', \gamma'') \in G^!(F). \]
	For all $\gamma \in G^!_{\mathrm{sreg}}(F)$ and $\delta \in G_{\mathrm{reg}}(F)$ such that $\gamma \leftrightarrow \delta$ together with $\tilde{\delta} \in \tilde{G}^{(2)}$ that maps to $\delta$, we have
	\[ s^!_c(\gamma) \Delta^{\bpsi_c}(\gamma, \tilde{\delta}) = (c, -1)_{F, 2}^{n''} \Delta^{\bpsi}(\gamma, \tilde{\delta}). \]
\end{proposition}
\begin{proof}
	Let $\delta$ be parametrized by $(K, K^{\natural}, x, c)$. Take the decomposition $K = K' \times K''$ and so forth, as in \S\ref{sec:calibration}. Choose $a'' = (a_i)_{i \in I''} \in (K'')^{\times}$ in the following way:
	\begin{itemize}
		\item if $K_i$ is a field, say $K_i = K^{\natural}_i(\sqrt{D_i})$, we take $a_i = \sqrt{D_i}$;
		\item if $K_i \simeq K^{\natural}_i \times K^{\natural}_i$, we take $a_i = (1, -1)$ and set $D_i = 1$.
	\end{itemize}
	Thus $\tau(a'') = -a''$.
	
	Also take $\omega = (\omega', \omega'') \in K^{\times}$ with $\omega/\tau(\omega) = x$. Claim:
	\begin{equation}\label{eqn:Delta-var-0}
		s^!_c(\gamma) = \left(c, N_{K'|F}(\omega') \right)_{F, 2} \left(c, N_{K''|F}(a''\omega'') \right)_{F, 2}.
	\end{equation}
	
	To obtain \eqref{eqn:Delta-var-0}, we use the data above and Lemma \ref{prop:SN-formula} to describe  $\mathrm{SN}(\gamma')$ and $\mathrm{SN}(\gamma'')$, with due care on the $-1$ twist in the correspondence of conjugacy classes; that twist is responsible for the $a''$ factor.
	
	Next, in the notation of Lemma \ref{prop:Delta-cocycle}, we claim that
	\begin{equation}\label{eqn:Delta-var-1}
		\left( c, N_{K''|F}(a'') \right)_{F, 2} \sgn''(c) = (c, -1)_{F, 2}^{n''}.
	\end{equation}
	
	To obtain \eqref{eqn:Delta-var-1}, set $I''_0 := \left\{i \in I'': K_i\;\text{is a field} \right\}$. The choice of $a''$ implies
	\begin{align*}
		\left(c, N_{K''|F}(a'') \right)_{F, 2}
		& = \prod_{i \in I''} \left(c, N_{K^{\natural}_i | F}(-D_i) \right)_{F, 2} \\
		& = \prod_{i \in I''} (c, -D_i)_{K^{\natural}_i, 2}
		= \prod_{i \in I''} (c, -1)_{K^{\natural}_i, 2} \cdot \prod_{i \in I''_0} \sgn_{K_i | K^{\natural}_i}(c) \\
		& = \prod_{i \in I''} (c, -1)_{F, 2}^{[K^{\natural}_i : F]} \cdot \sgn''(c)
		= (c, -1)_{F, 2}^{n''} \sgn''(c)
	\end{align*}
	since $\sum_{i \in I''} [K^{\natural}_i: F] = \frac{1}{2} \sum_{i \in I''} [K_i:F] = n''$. Standard properties of Hilbert symbols are used in the above.
		
	Since $N_{K|F}(\omega) = N_{K'|F}(\omega') N_{K''|F}(\omega'')$, we now obtain
	\begin{align*}
		s^!_c(\gamma) \Delta^{\bpsi_c}(\gamma, \tilde{\delta}) & = s^!_c(\gamma) \Delta^{\bpsi}(\gamma, \Ad(g)(\tilde{\gamma}))  \quad \text{(by Lemma \ref{prop:Ad-transport})} \\
		& = s^!_c(\gamma) \left(c, N_{K|F}(\omega) \right)_{F, 2} \Delta^{\bpsi}(\gamma, \mathrm{CAd}(g)(\tilde{\delta})) \quad \text{(by Lemma \ref{prop:st-conj})} \\
		& = s^!_c(\gamma) \left(c, N_{K|F}(\omega) \right)_{F, 2} \sgn''(c) \Delta^{\bpsi}(\gamma, \tilde{\delta}) \quad \text{(by Lemma \ref{prop:Delta-cocycle})} \\
		& \stackrel{\text{\eqref{eqn:Delta-var-0}}}{=} \left(c, N_{K'|F}(\omega') \right)_{F, 2} \left(c, N_{K''|F}(a''\omega'') \right)_{F, 2} \left(c, N_{K|F}(\omega) \right)_{F, 2} \sgn''(c) \Delta^{\bpsi}(\gamma, \tilde{\delta}) \\
		& = (c, N_{K''|F}(a''))_{F, 2} \sgn''(c) \Delta^{\bpsi}(\gamma, \tilde{\delta}) \\
		& \stackrel{\text{\eqref{eqn:Delta-var-1}}}{=} (c, -1)_{F, 2}^{n''} \Delta^{\bpsi}(\gamma, \tilde{\delta}),
	\end{align*}
	as desired.
\end{proof}

Next, restriction of functions induces an isomorphism $C^\infty_{c, \asp}(\tilde{G}) \rightiso C^\infty_{c, \asp}(\tilde{G}^{(2)})$ between spaces of anti-genuine $C^\infty_c$-functions, preserving orbital integrals. We define $\Trans_{\mathbf{G}^!, \tilde{G}^{(2)}}$ and $\trans_{\mathbf{G}^!, \tilde{G}^{(2)}}$ accordingly, with an exponent to indicate their dependence on additive characters.

\begin{theorem}\label{prop:Trans-var}
	Given $c \in F^{\times}$, define $\zeta, \Upsilon^\zeta, \Upsilon_\zeta$ as before. For each $\mathbf{G}^! \in \Endo_{\elli}(\tilde{G})$ corresponding to $(n', n'') \in \Z_{\geq 0}^2$, we have
	\begin{align*}
		\Upsilon^\zeta \circ \Trans^{\bpsi_c}_{\mathbf{G}^!, \tilde{G}^{(2)}} & = (c, -1)_{F, 2}^{n''} \Trans^{\bpsi}_{\mathbf{G}^!, \tilde{G}^{(2)}}, \\
		\trans^{\bpsi_c}_{\mathbf{G}^!, \tilde{G}^{(2)}} \circ \Upsilon_\zeta & = (c, -1)_{F, 2}^{n''} \trans^{\bpsi}_{\mathbf{G}^!, \tilde{G}^{(2)}}.
	\end{align*}
\end{theorem}
\begin{proof}
	It suffices to prove the first equality. The second one follows by dualization.
	
	Given $f \in \orbI_{\asp}(\tilde{G}^{(2)}) \otimes \mes(G)$ and $\gamma \in G^!_{\mathrm{sreg}}(F)$, the definitions of $\Upsilon^\zeta$, transfer and Proposition \ref{prop:Delta-var} imply
	\begin{align*}
		S_{G^!}\left( \gamma, \Upsilon^\zeta \Trans^{\bpsi_c}_{\mathbf{G}^!, \tilde{G}^{(2)}}(f)\right)
		& = s^!_c(\gamma) S_{G^!}\left( \gamma, \Trans^{\bpsi_c}_{\mathbf{G}^!, \tilde{G}^{(2)}}(f)\right) \\
		& = s^!_c(\gamma) \sum_{\delta: \gamma \leftrightarrow \delta} \Delta^{\bpsi_c}(\gamma, \tilde{\delta}) I_{\tilde{G}^{(2)}}(\tilde{\delta}, f) \\
		& = (c, -1)_{F, 2}^{n''} \sum_{\delta: \gamma \leftrightarrow \delta} \Delta^{\bpsi}(\gamma, \tilde{\delta}) I_{\tilde{G}^{(2)}}(\tilde{\delta}, f) \\
		& = (c, -1)_{F, 2}^{n''} S_{G^!}\left( \gamma, \Trans^{\bpsi}_{\mathbf{G}^!, \tilde{G}^{(2)}}(f) \right),
	\end{align*}
	where $\tilde{\delta}$ is any preimage of $\delta$ in $\tilde{G}^{(2)}$. Since $f$ and $\gamma$ are arbitrary, the desired equality follows.
\end{proof}
	
\subsection{Proof of the Theorem \ref{prop:GS}}\label{sec:proof-GS}
A short proof of Theorem \ref{prop:GS} can now be given.

\begin{proof}[Proof of the Theorem \ref{prop:GS}]
	Given $\phi \in \Phi_{\mathrm{bdd}}(\tilde{G})$ and $\chi \in \EuScript{S}_\phi^\vee$, consider the $T_{\phi, s} \in D_-(\tilde{G}) \otimes \mes(G)^\vee$ in \eqref{eqn:T}; it depends only on the image $x \in \EuScript{S}_\phi$ of $s$ by Lemma \ref{prop:T}. Set
	\[ {}^{\star} \pi_{\phi, \chi} := |\EuScript{S}_\phi|^{-1} \sum_{x \in \EuScript{S}_\phi} \chi(x) T_{\phi, x}.  \]
	
	We also view it as an element of $D_-(\tilde{G}^{(2)}) \otimes \mes(G)^\vee$ and denote it by ${}^{\star} \pi_{\phi, \chi}^{\bpsi}$ to emphasize the dependence on $\bpsi$. Claim:
	\[ {}^{\star} \pi^{\bpsi_c}_{\phi\zeta, \chi \delta_c} = {}^{\star} \pi^{\bpsi}_{\phi, \chi}. \]
	Indeed, this is the special case of \cite[Proposition 9.2.3]{Li24a} for bounded L-parameters. The second equality in Theorem \ref{prop:Trans-var} serves as the main input in proving the cited result, thus it does not use the assertion we seek.
	
	On the other hand, Luo's Theorem \ref{prop:Luo-endo} implies that ${}^{\star} \pi^{\bpsi}_{\phi, \chi} = \Tr (\pi^{\bpsi}_{\phi, \chi})$, and ditto for ${}^{\star} \pi^{\bpsi_c}_{\phi\zeta, \chi\delta_c}$. This concludes the proof.
\end{proof}

The arguments above rely on Luo's work \cite{Luo20}, which in turn is based on the Gan--Ichino multiplicity formula \cite[Theorem 1.4]{GI18}. None of these references depend on \cite[Theorem 12.1]{GS1}, so there is no circularity here.

\subsection{Remarks on L-groups}\label{sec:remark-L}
We begin by describing the action of the similitude group on L-packets.

\begin{proposition}\label{prop:Adg-rep}
	Consider $g \in \GSp(W)$ with $c := \nu(g)$. For all $\phi \in \Phi_{\mathrm{bdd}}(\tilde{G})$ and $\chi \in \EuScript{S}_\phi^\vee$, the isomorphism $\Ad(g): \tilde{G}^{\bpsi_c} \rightiso \tilde{G}^{\bpsi}$ in \eqref{eqn:Adg-lift} transports $\pi^{\bpsi}_{\phi, \chi}$ to $\pi^{\bpsi_c}_{\phi, \chi}$.
	
	Let $\zeta = \zeta_c: \Weil{F} \to \bmu_2$ be the character associated with the coset $cF^{\times 2}$ in $F^{\times}$, and $\delta_c$ be as in Definition \ref{def:delta-c}. If we work over $\tilde{G}^{(2)}$, then $\Ad(g)$ transports  $\pi^{\bpsi}_{\phi, \chi}$ to $\pi^{\bpsi}_{\phi\zeta, \chi \delta_c}$ up to isomorphism.
\end{proposition}
\begin{proof}
	Define $T^{\bpsi}_{\phi, s}$ and $T^{\bpsi_c}_{\phi, s}$ as in \eqref{eqn:T}, the exponent indicating their dependence on additive characters. In view of Theorem \ref{prop:Trans-var}, a transport of structure in endoscopy via $g: (W, c\lrangle{\cdot|\cdot}) \rightiso (W, \lrangle{\cdot|\cdot})$, and the fact that $\epsilon(\phi^{s = -1})$ does not depend on $\bpsi$, we see $T^{\bpsi}_{\phi, s}$ is transported to $T^{\bpsi_c}_{\phi, s}$ by $\Ad(g)$. The characterization in Theorem \ref{prop:Luo-endo} implies that the same holds for $\pi^{\bpsi}_{\phi, \chi}$ and $\pi^{\bpsi_c}_{\phi, \chi}$, up to isomorphism.
	
	The second assertion follows from the first one and Theorem \ref{prop:GS}.
\end{proof}

In fact, the above holds when $\phi$ is generalized to an Arthur parameter (see \cite{Li24a}), with the same proof.

\begin{remark}\label{rem:Adg-rep}
	The action of $\GSp(W)$ on $G$ descends to the adjoint group $G_{\mathrm{ad}}(F)$. For quasi-split connected reductive groups, such an action is expected to preserve L-parameter and shift the character of component group in a precise way, see \cite[\S 9]{GGP1}. In our case,
	\begin{itemize}
		\item the L-parameter gets shifted by $\zeta = \zeta_c$,
		\item the shift for characters of $\EuScript{S}_\phi$ in \textit{loc.\ cit.}\ no longer works. 
	\end{itemize}
	
	The first shift can be explained by Weissman's formalism \cite{Weis18}. This is done in \cite[Theorem 11.1]{GG}, which we rephrase below. What is canonically defined is just the L-group $\Lgrp{\tilde{G}^{(2)}}$. By \cite[Lemma 5.2.1]{Li20}, splittings $\Lgrp{\tilde{G}^{(2)}} \simeq \tilde{G}^\vee \times \Weil{F}$ exist, but they depend on the datum $(W, \lrangle{\cdot|\cdot})$ where $\lrangle{\cdot|\cdot}$ is taken up to $F^{\times 2}$; equivalently, they depend on $G(F)$-conjugacy classes of $F$-pinnings.
	
	After applying $\Ad(g)$, the datum $(W, \lrangle{\cdot|\cdot})$ is changed to $(W, c\lrangle{\cdot|\cdot})$. By \cite[Lemma 5.2.2]{Li20}, the splittings of $\Lgrp{\tilde{G}^{(2)}}$ differ by
	\[ \tilde{G}^\vee \times \Weil{F} \rightiso \tilde{G}^\vee \times \Weil{F}, \quad (\check{g}, w) \mapsto (\check{g}\zeta(w), w) . \]
	Hence $\zeta$ appears naturally in Proposition \ref{prop:Adg-rep} if one uses the dual group instead of the L-group.
\end{remark}

Hereafter, restrict $\Ad(g)$ to $\tilde{G}^{(2)}$ in \eqref{eqn:Adg-lift} and assume $c = -1$. This yields the well-known MVW involution on $\tilde{G}^{(2)}$ that transports every genuine irreducible representation $\pi$ of $\tilde{G}^{(2)}$ to its contragredient $\pi^\vee$, up to isomorphism; see \cite[p.36, p.92]{MVW87}.

\begin{corollary}\label{prop:contragredient}
	For all $\phi \in \Phi_{\mathrm{bdd}}(\tilde{G})$ and $\chi \in \EuScript{S}_\phi^\vee$, we have
	\begin{equation*}
		\left(\pi^{\bpsi}_{\phi, \chi}\right)^\vee \simeq \pi^{\bpsi}_{\phi\zeta_{-1}, \chi\delta_{-1}}.
	\end{equation*}
\end{corollary}
\begin{proof}
	Combine the discussion above with Proposition \ref{prop:Adg-rep}.
\end{proof}

A description of contragredient representations in terms of enhanced L-parameters $(\phi, \chi)$ for a quasi-split connected reductive group $R$ is proposed by D.\ Prasad in \cite[Conjecture 2]{Pra18}; it involves the Chevalley involution $c_{R^\vee}$ on the dual side. Note that $c_{\Sp(2n, \CC)} = \identity$; see p.5 of \textit{loc.\ cit.}

As in Remark \ref{rem:Adg-rep}, Prasad's recipe for characters of $\EuScript{S}_\phi$ cannot carry over to $\tilde{G}^{(2)}$. For the L-parameter of the contragredient, he predicts that it differs from the original one by the Chevalley involution $\Lgrp{c}$ of the L-group.

The idea here is that contragredients live on the opposite twofold covering determined by $(W, -\lrangle{\cdot|\cdot})$, although that is uniquely isomorphic to $\tilde{G}^{(2)}$. The appropriate definition of $\Lgrp{c}$ for $\tilde{G}^{(2)}$ should render
\begin{equation*}\label{eqn:Lgrp-c}
	\begin{tikzcd}
		\Lgrp{\tilde{G}}^{(2)} \arrow[d, "a"'] \arrow[r, "{\Lgrp{c}}"] & \Lgrp{\tilde{G}}^{(2)} \arrow[d, "b"] \\
		\Sp(2n, \CC) \times \Weil{F} \arrow[r, "{(c_{\Sp(2n, \CC)} = \identity, \identity)}"' inner sep=0.8em] & \Sp(2n, \CC) \times \Weil{F}
	\end{tikzcd} 
\end{equation*}
commutative, where $a$ (resp.\ $b$) is the L-isomorphism associated with $(W, \lrangle{\cdot|\cdot})$ (resp.\ $(W, -\lrangle{\cdot|\cdot})$).

As seen earlier, if one splits the L-group via $a$ everywhere, without using $b$, then
\[ \Lgrp{c}(\check{g}, w) = (\check{g}\zeta_{-1}(w), w). \]
This agrees with Corollary \ref{prop:contragredient}.

\printbibliography[heading=bibintoc]

\vspace{1em}
\begin{flushleft} \small
	W.-W. Li: Beijing International Center for Mathematical Research / School of Mathematical Sciences, Peking University. No.\ 5 Yiheyuan Road, Beijing 100871, People's Republic of China. \\
	E-mail address: \href{mailto:wwli@bicmr.pku.edu.cn}{\texttt{wwli@bicmr.pku.edu.cn}}
\end{flushleft}

\end{document}